\documentclass[11pt,reqno]{amsart}
\title{Heisenberg Uncertainty Inequality for Gabor Transform}
\usepackage{amsmath,amssymb,amsthm,amsrefs,mathrsfs}
\numberwithin{equation}{section}
\usepackage{enumerate}
\usepackage{tikz}
\usetikzlibrary{fit,shapes,snakes,positioning,calc,patterns}
\pgfdeclarelayer{background}
\pgfdeclarelayer{foreground}
\pgfsetlayers{background,main,foreground} 
 
\usepackage{tikz}
\usepackage{tabularx}
\usepackage{cellspace}
\theoremstyle{definition}
\newtheorem{thm}{\sc Theorem}[section]

\newtheorem{lem}[thm]{\sc Lemma}

\newtheorem{prop}[thm]{\sc Proposition}

\newtheorem{ex}[thm]{\sc Example}
\newcommand{\R}{\mathbb{R}}

\newcommand{\h}{\mathbb{H}}

\newcommand{\C}{\mathbb{C}}
\newcommand{\F}{\mathscr{F}}
\newcommand{\g}{\mathfrak{g}}

\newcommand{\hf}{\mathfrak{h}}
\newcommand{\z}{\mathfrak{z}}

\newcommand{\U}{\mathcal{U}}
\newcommand{\W}{\mathcal{W}}
\newcommand{\B}{\mathcal{B}}
\newcommand{\ch}{\mathcal{H}}
\newcommand{\hh}{\mathscr{H}}

\newcommand{\CS}{\mathcal{S}}

\DeclareMathOperator*{\tr}{tr}
\DeclareMathOperator*{\expo}{exp}

\DeclareMathOperator*{\pf}{Pf}

\DeclareMathOperator*{\ind}{ind}
\DeclareMathOperator*{\Aut}{Aut}
\allowdisplaybreaks
\begin{document}
\begin{abstract}
We discuss Heisenberg uncertainty inequality for groups of the form $K \ltimes \R^n$, $K$ is a separable unimodular locally compact group of type I. This inequality is also proved for Gabor transform for several classes of groups of the form $K \ltimes \R^n$. 
\end{abstract}
\author[A. Bansal]{ASHISH BANSAL}
\address{Department of Mathematics, Keshav Mahavidyalaya (University of Delhi), H-4-5 Zone, Pitampura, Delhi, 110034, India.}
\email{abansal@keshav.du.ac.in}
\author[A. Kumar]{AJAY KUMAR}
\address{Department of Mathematics, University of Delhi, Delhi, 110007, India.}
\email[Corresponding author]{akumar@maths.du.ac.in}
\maketitle
\section{Introduction}
The uncertainty principle states that a non-zero function and its Fourier transform cannot both be sharply localized. The most precise way of formulating this principle quantitatively is the inequality known as \textit{Heisenberg uncertainty inequality}. Let $f$ be any function in $L^2(\R)$. The Fourier transform of $f$ is defined as
\begin{flalign*}
&&\widehat{f}(\omega)&=\displaystyle\int_{\R^n}{f(x)\ e^{-2\pi i \omega x}}\ dx.&
\end{flalign*}
The following theorem gives the Heisenberg uncertainty inequality for the Fourier transform on $\R$:
\begin{thm}\label{theq-R} 
For any $f \in L^2(\R)$, we have
\begin{flalign}
&&\dfrac{\|f\|_2^2}{4\pi}&\leq \left(\displaystyle\int_{\R}{x^2\ |f(x)|^2}\ dx\right)^{1/2}\left(\displaystyle\int_{\R}{\omega^2\ |\widehat{f}(\omega)|^2}\ d\omega\right)^{1/2},   \label{heq-R}&
\end{flalign}
where $\|\cdot\|_2$ denotes the $L^2$-norm. 
\end{thm}
\noindent For proof of the theorem, refer to \cite{Fol:Sit:97}.

The representation of $f$ as a function of $x$ is usually called its \textit{time-representation}, while the representation of $\hat{f}$ as a function of $\omega$ is called its \textit{frequency-representation}. The Fourier transform has been the most commonly used tool for analyzing frequency properties of a given signal, but the problem with this tool is that after transformation, the information about time is lost and it is hard to tell where a certain frequency occurs. To counter this problem, we can use \textit{joint time-frequency representation}, i.e., Gabor transform. 

Let $\psi \in L^2(\R)$ be a fixed non-zero function usually called a \textit{window function}. The Gabor transform of a function $f\in L^2(\R)$ with respect to the window function $\psi$ is defined by
\begin{flalign*}
G_\psi f : \R \times \widehat{\R} \rightarrow \C
\end{flalign*}
such that
\begin{flalign*}
G_\psi f(t,\omega)=\int_{\R}{f(x)\ \overline{\psi(x-t)}\ e^{-2\pi i \omega x}}\ dx,
\end{flalign*}
for all $(t,\omega)\in \R \times \widehat{\R}$. The following uncertainty inequality of Heisenberg-type has been proved by Wilczok \cite{Wil:00}. 
\begin{thm}\label{theq-GT-R} 
Let $\psi$ be a window function. Then, for arbitrary $f \in L^2(\R)$, the following inequality holds
\begin{flalign}
\dfrac{\|\psi\|_2\ \|f\|_2^2}{4\pi}&\leq\left(\displaystyle\int_{\R}{x^2\ |f(x)|^2}\ dx\right)^{1/2}\left(\displaystyle\int_{\R^2}{\omega^2\ |G_\psi f(t,\omega)|^2}\ dt\ d\omega\right)^{1/2}. \label{heq-GT-R}&
\end{flalign}
\end{thm}
 
The continuous Gabor transform for second countable, non-abelian, unimodular and type I groups has been defined by Farashahi and Kamyabi-Gol in \cite{Far:Kam:12}. 

In section $2$, we shall state the Heisenberg uncertainty inequality for the groups of the form $K\ltimes \R^n$, where $K$ is a separable unimodular locally compact group of type I and prove it for the semi-direct product $K \ltimes \R^n$ (where $K$ is a compact subgroup of the group of automorphisms of $\R^n$). In section $3$, we shall discuss continuous Gabor transform and prove Heisenberg uncertainty inequality for Gabor transform on $K\ltimes \R^n$ (where $K$ is a separable unimodular locally compact group of type I) that satisfy the Heisenberg uncertainty inequality for Fourier transform. The explicit forms of Heisenberg uncertainty inequality for Gabor transform are obtained for $K \ltimes \R^n$, $K$ is a compact subgroup of $\Aut(\R^n)$; $\R^n\times K$, $K$ is separable unimodular locally compact group of type~I; Heisenberg group $\h_n$; Thread-like nilpotent Lie groups; $2$-NPC nilpotent Lie groups and several classes of connected, simply connected nilpotent Lie groups.
\section{Extensions of $\R^n$}
Let $G=K\ltimes \R^n$, where $K$ is a separable unimodular locally compact group of type I. For $\gamma \in \widehat{\R^n}$, let $G_\gamma$, $K_\gamma$ denote the stabilizer subgroup of $\gamma$ in $G$ and $K$ respectively and let
\begin{flalign*}
\check{G}_\gamma=\{\nu\in \widehat{G_\gamma}: \nu|_{\R^n}\ \text{is a finite multiple of}\ \gamma\}.
\end{flalign*}
Then for $\nu \in \check{G}_\gamma$, the representation $\pi_\nu=\ind_{G_\gamma}^G \nu$ is irreducible and
\begin{flalign*}
\widehat{G}=\cup_{\widehat{\R^n}/G}\{\pi_\nu :\nu \in \check{G}_\gamma\}.
\end{flalign*}
Since $\R^n$ is abelian, any $\nu \in \check{G}_\gamma$ is of the form $\nu=\sigma \otimes \gamma$, $\nu(kx)=\sigma(k)\gamma(x)$, $k\in K_\gamma$, $x\in \R^n$ and $\sigma \in \widehat{K_\gamma}$.

We consider the induced representations
\begin{flalign*}
\pi_{\gamma,\sigma}={\ind}_{G_\gamma}^{G}(\gamma \otimes \sigma).
\end{flalign*}
The Plancherel formula for $G$ (for details, see \cite{Kle:Lip:73}) takes the following form:
\begin{prop}[Plancherel formula]
For all $f\in L^2(G)$, we have
\begin{flalign*}
\displaystyle\int_{G}{|f(g)|^2}\ dg=\displaystyle\int_{\widehat{\R^n}/G} \displaystyle\int_{\widehat{K_\gamma}}{\|\pi_{\gamma,\sigma}(f)\|_2^2}\ d\mu_\gamma(\sigma)\ d\overline{\mu}_{\R^n}(\overline{\gamma}).
\end{flalign*} 
\end{prop}

We now state the Heisenberg uncertainty inequality for $G$ which has been proved, in particular cases of $\R^n$ (see \cite{Fol:Sit:97}); Heisenberg group (see \cite{Tha:90},\cite{Sit:Sun:Tha:95} and \cite{Xia:He:12}); $\R^n \times K$ (where $K$ is a separable unimodular locally compact group of type I), Euclidean motion group $M(n)=SO(n)\ltimes \R^n$ and several general classes of nilpotent Lie groups which include thread-like nilpotent Lie groups, $2$-NPC nilpotent Lie groups and several low-dimensional nilpotent Lie groups (see \cite{Ban:Kum:15}). 
\begin{thm}\label{theq-FT-G}
For any $f \in L^2(G)$ and $a,b \geq 1$, we have
\begin{flalign}
\|f\|_2^{\left(\frac{1}{a}+\frac{1}{b}\right)} &\leq C\left(\displaystyle\int_{K\times \R^n}{\|x\|^{2a}\ |f(k,x)|^2}\ dx\ dk\right)^{\frac{1}{2a}}\nonumber&\\
&\quad \times\left(\displaystyle\int_{\widehat{\R^n}/ G}\displaystyle\int_{\widehat{K_\gamma}}{\|\gamma\|^{2b}\|\pi_{\gamma,\sigma}(f)\|_{\text{HS}}^2}\ d\mu_\gamma(\sigma)\ d\overline{\mu}_{\R^n}(\overline{\gamma}),\right)^{\frac{1}{2b}},\label{heq-FT-G}&\tag{H}
\end{flalign}
where $C$ is a constant.
\end{thm}

We do not know whether the inequality \eqref{heq-FT-G} is true for $K \ltimes \R^n$, however we now prove the Heisenberg uncertainty inequality when $K$ is a compact subgroup of $\Aut(\R^n)$.

Let $G$ be the semi-direct product $K \ltimes \R^n$, where $K$ is a compact subgroup of $\Aut(\R^n)$. The Haar measure on $G$ is $dg=d\nu(k)\ dx$, where $d\nu(k)$ denotes the normalized Haar measure of $K$ and $dx$ denotes the Lebesgue measure on $\R^n$. We shall now give more explicit description of the unitary dual space of the group $G$ in this case which can be determined by Mackey's theory. For more details, refer to \cite{Mac:76}.  

Let $\ell$ be a non-zero real linear form on $\R^n$ and let $\chi_\ell$ be the unit character of $\R^n$ defined by $\chi_\ell(x)=e^{i\langle l,x\rangle}$. The natural action $g\cdot \ell$ of $G$ on the dual vector space of $\R^n$ is given by 
\begin{flalign*}
\langle g\cdot \ell,x\rangle =\langle\ell, g^{-1}xg\rangle,
\end{flalign*}
for $g\in G$ and $x \in \R^n$. Therefore, if $g$ acts on $\widehat{\R^n}$ by 
\begin{flalign*}
g \cdot \chi_\ell (x):=\chi_\ell(g^{-1}xg),
\end{flalign*}
we get $g\cdot \chi_\ell =\chi_{g\cdot \ell}$. Define
\begin{flalign*}
K_\ell=\{k\in K : k\cdot \chi_\ell=\chi_\ell\}.
\end{flalign*}
Then, the subgroup $K_\ell \ltimes R^n$ is the stabilizer of $\chi_\ell$ in $G$. We take the normalized Haar measure $d\nu_\ell$ on $K_\ell$ and a normalized $K$-invariant measure $d\dot{\nu}_\ell$ on $K/K_\ell$ so that 
\begin{flalign*}
\displaystyle\int_K{\xi(k)}\ d\nu(k)=\displaystyle\int_{K/K_\ell}\displaystyle\int_{K_\ell}{\xi(kk')}\ d\nu_\ell(k')\ d\dot{\nu}_\ell (kK_\ell).
\end{flalign*}

Regarding the action of $K$ on $\widehat{\R^n}$ which is isomorphic to $\R^n$, we set by $d\bar{\ell}$ the image of the Lebesgue measure on $\R^n/K$ by the canonical projection $\R^n \ni \ell \mapsto \bar{\ell}:=K.\ell \in \R^n/K$ such that
\begin{flalign*}
\displaystyle\int_{\R^n}{\varphi(\ell)}\ d\ell=\displaystyle\int_{\R^n/K}\displaystyle\int_{K}{\varphi(k.\ell)}\ d\nu(k)\ d\bar{\ell}.
\end{flalign*}
Let $\sigma$ be an irreducible unitary representation of $K_\ell$ and $\hh_{\ell,\sigma}$ be the completion of the vector space of all continuous mapping $\xi: K\rightarrow \hh_\sigma$ which satisfies $\xi(ks)=\sigma(s)^\ast(\xi(k))$ for $k\in K$ and $s \in K_\ell$ with respect to the norm
\begin{flalign*}
\|\xi\|_2=\left(\displaystyle\int_{K}{\|\xi(k)\|_{\hh_\sigma}^2}\ d\nu(k)\right)^{1/2}.
\end{flalign*}
The induced representation
\begin{flalign*}
\pi_{\ell,\sigma}:={\ind}_{K_\ell \ltimes \R^n}^{G}{(\sigma \otimes \chi_\ell)},
\end{flalign*}
realized on the Hilbert space $\hh_{\ell,\sigma}$ by
\begin{flalign*}
\pi_{\ell,\sigma}(k,x)\xi(s)=e^{i\langle \ell,s^{-1}xs\rangle}\xi(k^{-1}s)=e^{i \langle s.\ell,x\rangle}\xi(k^{-1}s),
\end{flalign*}
for $\xi\in \hh_{\ell,\sigma}$, $(k,x)\in G$ and $s \in K$, is an irreducible representation of $G$. Furthermore, every infinite dimensional irreducible unitary representation of $G$ is equivalent to some representation $\pi_{\ell,\sigma}$.

The Plancherel formula \cite[Theorem 7.44]{Fol:94} can be stated in this particular case as follows:
\begin{prop}[Plancherel formula] \label{PT-CE}
Let $f \in L^1(G)\cap L^2(G)$, then
\begin{flalign}
&&\displaystyle\int_{K\times \R^n}{|f(k,x)|^2}\ dx\ dk &=\displaystyle\int_{\R^n/K}\sum_{\sigma\in \widehat{K}_\ell}{\|\pi_{\ell,\sigma}(f)\|_{\text{HS}}^2}\ d\bar{\ell}. &
\end{flalign}
\end{prop}

We shall now establish Heisenberg uncertainty inequality for Fourier transform on $G$. A particular case for the Euclidean motion group has been proved in \cite{Ban:Kum:15}.
\begin{thm}\label{theq-G}
For any $f \in L^2(G)$ and $a,b \geq 1$, we have
\begin{flalign}
\|f\|_2^{\left(\frac{1}{a}+\frac{1}{b}\right)} &\leq C\left(\displaystyle\int_{K\times \R^n}{\|x\|^{2a}\ |f(k,x)|^2}\ dx\ dk\right)^{\frac{1}{2a}}\nonumber&\\*
&\qquad \times\left(\displaystyle\int_{\R^n/ K}\sum_{\sigma\in \widehat{K_\ell}}{\|\ell\|^{2b}\|\pi_{\ell,\sigma}(f)\|_{\text{HS}}^2}\ d\bar{\ell}\right)^{\frac{1}{2b}},\label{heq-G}&
\end{flalign}
where $C$ is a constant.
\end{thm}
\begin{proof}
Define the norm $\|\cdot\|$ on $L^2(G)$ as
\begin{flalign*}
&&\|f\|:&=\left(\displaystyle\int_{K\times \R^n}{(1+\|x\|^{2a})\ |f(k,x)|^2}\ dx\ dk\right)^{1/2}&\\
&&&\quad +\left(\displaystyle\int_{\R^n/K}\sum_{\sigma\in \widehat{K_\ell}}{(1+\|\ell\|^{2b})\|\pi_{\ell,\sigma}(f)\|_{\text{HS}}^2}\ d\bar{\ell}\right)^{1/2}.&
\end{flalign*}
Then, the set $B=\{f\in L^2(G):\|f\|<\infty\}$ forms a Banach space which is contained in $L^2(G)$. If $0\neq f \in L^2(G)\setminus B$, then the right hand side of the inequality \eqref{heq-G} is always $+\infty$ and the inequality is trivially valid. 

Let $\CS(G)$ denote the space of $C^\infty$-functions which are rapidly decreasing on $G$. It can be shown that $\CS(G)$ is dense in $B$. Thus it suffices to prove the inequality \eqref{heq-G} for functions in $\CS(G)$. 

Let $f\in \CS(G)$. Assuming that both the integrals on right hand side of \eqref{heq-G} are finite, we have
\begin{flalign*}
&&\displaystyle\int_{\R^n}{|f(k,x)|^2}\ dx &< \infty,\ \text{for all}\ k\in K.&
\end{flalign*}
For $k \in K$, we define $f_k(x)=f(k,x)$, for every $x \in\R^n$. \\
Clearly, $f_k \in L^2(\R^n)$, for all $k\in K$. \\
Taking $x=(x_1,x_2,\ldots,x_n)$, $y=(y_1,y_2,\ldots,y_n)$ and proceeding as in the case of Euclidean motion group (see \cite[Theorem 2.2]{Ban:Kum:15}), we obtain
\begin{flalign}
\dfrac{\|f\|_2^2}{4\pi}&\leq\left(\displaystyle\int_{K\times \R^n}{\|x\|^{2a}\ |f(k,x)|^2}\ dx\ dk\right)^{\frac{1}{2a}}\left(\|f\|_2^2\right)^{\frac{1}{2}-\frac{1}{2a}} \nonumber&\\
&\qquad \times \left(\displaystyle\int_{K\times \R^n}{|y_1|^2\ |\widehat{f_k}(y)|^2}\ dy\ dk\right)^{1/2}. \label{step3-G}&
\end{flalign}
Now, using Plancherel formula on $\R^n$, we have
\begin{flalign}
&\displaystyle\int_{K\times \R^n}{|y_1|^2\ |\widehat{f_k}(y)|^2}\ dy\ dk \nonumber &\\
&=\displaystyle\int_{K\times \R^n}{|y_1|^2\ \left|\displaystyle\int_{\R^n}{f(k,x)\ e^{-2\pi i\langle{x,y}\rangle}}\ dx\right|^2}\ dy\ dk\nonumber &\\
&=\displaystyle\int_{K\times \R^n}{|y_1|^2\ |\F_{2,3,\ldots, n+1} f(k,y_1,y_2,\ldots,y_n)|^2}\ dy_1\ dy_2\ \ldots\ dy_n\ dk \nonumber &\\
&=\displaystyle\int_{K\times \R^n}{|y_1|^2\ |\F_{2} f(k,y_1,x_2,\ldots,x_n)|^2}\ dy_1\ dx_2\ \ldots\ dx_n\ dk, \label{step4-G}&
\end{flalign}
where $\F_i$ denotes the Fourier transform in the $i^{\text{th}}$ variable. \\
Since, $\dfrac{\partial f}{\partial x_1}\in \CS(G)$, we have 
\begin{flalign*}
&&\displaystyle\int_{\R}{\left|\dfrac{\partial f}{\partial x_1}(k,x_1,x_2,\ldots,x_n)\right|^2}\ dx_1&< \infty,&
\end{flalign*}
for all $k \in K$ and $x_i \in \R$ ($i=2,3,\ldots,n$).\\
So, $y_1\F_2 f(k,y_1,x_2,\ldots,x_n)\in L^2(\R)$ and
\begin{flalign*}
&&\left(\dfrac{\partial f}{\partial x_1}(k,x_1,x_2,\ldots,x_n)\right)^{\widehat{\ }}(y_1)&=2\pi i y_1\F_2 f(k,y_1,x_2,\ldots,x_n), &
\end{flalign*}
for all $k \in K$ and $x_i \in \R$ ($i=2,3,\ldots,n$). Then
\begin{flalign*}
\displaystyle\int_{\R}{|y_1|^2\ |\F_2 f(k,y_1,x_2,\ldots,x_n)|^2}\ dy_1&=\dfrac{1}{4\pi^2}\displaystyle\int_{\R}{\left|\dfrac{\partial f}{\partial x_1}(k,x_1,x_2,\ldots,x_n)\right|^2}\ dx_1. &
\end{flalign*}
Using Proposition \ref{PT-CE}, we have
\begin{flalign}
&\displaystyle\int_{K\times \R^n}{|y_1|^2\ |\F_2 f(k,y_1,x_2,\ldots,x_n)|^2}\ dy_1\ dx_2\ \ldots\ dx_n\ dk \nonumber&\\
&=\dfrac{1}{4\pi^2}\displaystyle\int_{K\times \R^n}{\left|\dfrac{\partial f}{\partial x_1}(k,x_1,x_2,\ldots,x_n)\right|^2}\ dx_1\ dx_2\ \ldots\ dx_n\ dk \nonumber&\\
&=\dfrac{1}{4\pi^2}\displaystyle\int_{\R^n/ K}\sum_{\sigma\in \widehat{K_\ell}}{\left\|\pi_{\ell,\sigma}\left(\dfrac{\partial f}{\partial x_1}\right)\right\|_{\text{HS}}^2}\ d\bar{\ell}. \label{step5-G}
\end{flalign}
Combining \eqref{step3-G}, \eqref{step4-G} and \eqref{step5-G}, we obtain
\begin{flalign*}
&&\dfrac{\|f\|_2^2}{2}&\leq\left(\displaystyle\int_{K\times \R^n}{\|x\|^{2a}\ |f(k,x)|^2}\ dx\ dk\right)^{\frac{1}{2a}}\left(\|f\|_2^2\right)^{\frac{1}{2}-\frac{1}{2a}}&\\
&&&\qquad \times\left(\displaystyle\int_{\R^n/ K}\sum_{\sigma\in \widehat{K_\ell}}{\left\|\pi_{\ell,\sigma}\left(\dfrac{\partial f}{\partial x_1}\right)\right\|_{\text{HS}}^2}\ d\bar{\ell}\right)^{1/2}, &
\end{flalign*}
which implies
\begin{flalign}
\dfrac{\|f\|_2^{1+\frac{1}{a}}}{2}&\leq\left(\displaystyle\int_{K\times \R^n}{\|x\|^{2a}\ |f(k,x)|^2}\ dx\ dk\right)^{\frac{1}{2a}}\left(\displaystyle\int_{\R^n/ K}\sum_{\sigma\in \widehat{K_\ell}}{\left\|\pi_{\ell,\sigma}\left(\dfrac{\partial f}{\partial x_1}\right)\right\|_{\text{HS}}^2}\ d\bar{\ell}\right)^{1/2}. \label{step6-G} &
\end{flalign}
For each non-zero linear form $\ell$ on $\R^n$ and each irreducible unitary representation $\sigma$ of $K_\ell$, consider the representation $\pi_{\ell,\sigma}$ realized on the Hilbert space $\hh_{\ell,\sigma}$ as
\begin{flalign*}
\pi_{\ell,\sigma}(k,x)\xi(s)=e^{i\langle \ell,s^{-1}xs\rangle}\xi(k^{-1}s)=e^{i \langle s.\ell,x\rangle}\xi(k^{-1}s),
\end{flalign*}
for $\xi\in \hh_{\ell,\sigma}$, $(k,x)\in G$ and $s \in K$. Since $f\in\CS(G)$, we observe that
\begin{flalign}
&\pi_{\ell,\sigma}\left(\dfrac{\partial f}{\partial x_1}\right)\xi(s)\nonumber&\\[4pt]
&=\displaystyle\int_{K\times \R^n}{\dfrac{\partial f}{\partial x_1}(k,x_1,x_2,\ldots,x_n)\ \pi_{\ell,\sigma}(k,x_1,x_2,\ldots,x_n)^\ast \xi(s)}\ dx_1\ dx_2\ \ldots\ dx_n\ dk \nonumber&\\[4pt]
&=\displaystyle\int_{K\times \R^n}{\lim_{h\rightarrow 0}{\left[\dfrac{f(k,x_1+h,x_2,\ldots,x_n)-f(k,x_1,x_2,\ldots,x_n)}{h}\right]}\ \pi_{\ell,\sigma}(k,x_1,x_2,\ldots,x_n)^\ast \xi(s)}\nonumber&\\*
&\hspace{200pt}\ dx_1\ dx_2\ \ldots\ dx_n\ dk \nonumber&\\
&=\lim_{h\rightarrow 0}\dfrac{1}{h}\left[\displaystyle\int_{K\times \R^n}{f(k,x_1+h,x_2,\ldots,x_n)\ \pi_{\ell,\sigma}(k,x_1,x_2,\ldots,x_n)^\ast \xi(s)}\ dx_1\ dx_2\ \ldots\ dx_n\ dk\right. \nonumber&\\[4pt]
&\qquad \left.-\displaystyle\int_{K\times \R^n}{f(k,x_1,x_2,\ldots,x_n)\ \pi_{\ell,\sigma}(k,x_1,x_2,\ldots,x_n)^\ast \xi(s)}\ dx_1\ dx_2\ \ldots\ dx_n\ dk\right] \nonumber&\\[4pt]
&=\lim_{h\rightarrow 0}\dfrac{1}{h}\left[\displaystyle\int_{K\times \R^n}{f(k,x_1,x_2,\ldots,x_n)\ \pi_{\ell,\sigma}(k,x_1-h,x_2,\ldots,x_n)^\ast \xi(s)}\ dx_1\ dx_2\ \ldots\ dx_n\ dk\right. \nonumber&\\
&\qquad \left.-\displaystyle\int_{K\times \R^n}{f(k,x_1,x_2,\ldots,x_n)\ \pi_{\ell,\sigma}(k,x_1,x_2,\ldots,x_n)^\ast \xi(s)}\ dx_1\ dx_2\ \ldots\ dx_n\ dk\right]. \label{step7-G}&
\end{flalign}
Let $e_1=\{1,0,0,\ldots,0\}\in \R^n$, then we can write
\begin{flalign*}
\pi_{\ell,\sigma}(k,x_1-h,x_2,\ldots,x_n)^\ast \xi(s)&=\pi_{\ell,\sigma}(k,x-he_1)^\ast \xi(s) &\\
&=e^{-i\langle \ell,s^{-1}(x-he_1)s\rangle}\ \xi(k^{-1}s) &\\
&=e^{i\langle \ell,s^{-1}(he_1)s\rangle}\ e^{-i\langle \ell,s^{-1}xs\rangle}\ \xi(k^{-1}s) &\\
&=e^{ih\langle \ell,s^{-1}e_1s\rangle}\ \pi_{\ell,\sigma}(k,x_1,x_2,\ldots,x_n)^\ast \xi(s). &
\end{flalign*}
Equation \eqref{step7-G} can be written as
\begin{flalign*}
&\pi_{\ell,\sigma}\left(\dfrac{\partial f}{\partial x_1}\right)\xi(s) &\\
&=\lim_{h\rightarrow 0}\left[\dfrac{e^{ih\langle \ell,s^{-1}e_1s\rangle}-1}{h}\right]\displaystyle\int_{K\times \R^n}{f(k,x_1,x_2,\ldots,x_n)}&\\
&\hspace{140pt}\times {\pi_{\ell,\sigma}(k,x_1,x_2,\ldots,x_n)^\ast \xi(s)}\ dx_1\ dx_2\ \ldots\ dx_n\ dk &\\[4pt]
&=\lim_{h\rightarrow 0}\left[\dfrac{e^{ih\langle \ell,s^{-1}e_1s\rangle}-1}{h}\right]\pi_{\ell,\sigma}(f)\xi(s) &\\*
&=i\langle \ell,s^{-1}e_1s\rangle\ \pi_{\ell,\sigma}(f)\xi(s). &
\end{flalign*}
Since $s\mapsto s^{-1}e_1s$ is a continuous map from $K$ to $\R^n$, so $\{s^{-1}e_1s : s\in K\}$ is bounded. For any orthonormal basis $\{\xi_j\}$ of $\hh_{\ell,\sigma}$, we have
\begin{flalign*}
\left\|\pi_{\ell,\sigma}\left(\dfrac{\partial f}{\partial x_1}\right)\right\|_{\text{HS}}^2 &=\sum_{j}\displaystyle\int_K{|i\langle \ell,s^{-1}e_1s\rangle\ \pi_{\ell,\sigma}(f)\xi_j(s)|^2}\ ds &\\
&\leq const.\ \|\ell\|^2\sum_{j}\displaystyle\int_K{|\pi_{\ell,\sigma}(f)\xi_j(s)|^2}\ ds =const.\ \|\ell\|^2\|\pi_{\ell,\sigma}(f)\|_{\text{HS}}^2. &
\end{flalign*}
So, \eqref{step6-G} can be written as
\begin{flalign}
\|f\|_2^{1+\frac{1}{a}}&\leq C\left(\displaystyle\int_{K\times \R^n}{\|x\|^{2a}\ |f(k,x)|^2}\ dx\ dk\right)^{\frac{1}{2a}}\left(\displaystyle\int_{\R^n/ K}\sum_{\sigma\in \widehat{K_\ell}}{\|\ell\|^2\left\|\pi_{\ell,\sigma}(f)\right\|_{\text{HS}}^2}\ d\bar{\ell}\right)^{1/2}. \label{step8-G}&
\end{flalign}
Using H\"older's inequality, we have
\begin{flalign*}
&\left(\displaystyle\int_{\R^n/ K}\sum_{\sigma\in \widehat{K_\ell}}{\|\ell\|^{2b}\|\pi_{\ell,\sigma}(f)\|_{\text{HS}}^2}\ d\bar{\ell}\right)^{\frac{1}{b}}\left(\displaystyle\int_{\R^n/ K}\sum_{\sigma\in \widehat{K_\ell}}{\|\pi_{\ell,\sigma}(f)\|_{\text{HS}}^2}\ d\bar{\ell}\right)^{1-\frac{1}{b}}&\\
&\geq  \displaystyle\int_{\R^n/ K}\sum_{\sigma\in \widehat{K_\ell}}{\|\ell\|^2\ \|\pi_{\ell,\sigma}(f)\|_{\text{HS}}^2}\ d\bar{\ell}, &
\end{flalign*}
which implies
\begin{flalign}
&\displaystyle\int_{\R^n/ K}\sum_{\sigma\in \widehat{K_\ell}}{\|\ell\|^2\ \|\pi_{\ell,\sigma}(f)\|_{\text{HS}}^2}\ d\bar{\ell}& \leq \left(\displaystyle\int_{\R^n/ K}\sum_{\sigma\in \widehat{K_\ell}}{\|\ell\|^{2b}\|\pi_{\ell,\sigma}(f)\|_{\text{HS}}^2}\ d\bar{\ell}\right)^{\frac{1}{b}}\left(\|f\|_2^2\right)^{1-\frac{1}{b}}.   \label{step9-G}&
\end{flalign}
Combining \eqref{step8-G} and \eqref{step9-G}, we obtain
\begin{flalign*}
\|f\|_2^{\left(\frac{1}{a}+\frac{1}{b}\right)}&\leq C\left(\displaystyle\int_{K\times \R^n}{\|x\|^{2a}\ |f(k,x)|^2}\ dx\ dk\right)^{\frac{1}{2a}}\left(\displaystyle\int_{\R^n/ K}\sum_{\sigma\in \widehat{K_\ell}}{\|\ell\|^{2b}\|\pi_{\ell,\sigma}(f)\|_{\text{HS}}^2}\ d\bar{\ell}\right)^{\frac{1}{2b}}.& 
\end{flalign*}
\end{proof}
\section{Continuous Gabor Transform}
Let $\ch$ be a separable Hilbert space and let $\B(\ch)$ denotes the set of all bounded linear operators on $\ch$. An operator $T\in \B(\ch)$ is called Hilbert-Schmidt operator if and only if 
\begin{flalign*}
\sum_{k}{\|Te_k\|^2}<\infty,
\end{flalign*}
for some, and hence for any, orthonormal basis $\{e_k\}$ of $\ch$. We denote the set of all Hilbert-Schmidt operators on $\ch$ by $\text{HS}(\ch)$. For each $T\in \text{HS}(\ch)$, the Hilbert-Schmidt norm $\|T\|_{\text{HS}}$ of $T$ is defined as
\begin{flalign*}
\|T\|_{\text{HS}}^2:=\sum_{k}{\|Te_k\|^2}.
\end{flalign*}
Also, $\text{HS}(\ch)$ forms a Hilbert space with the inner product given by
\begin{flalign*}
\langle T,S\rangle_{\text{HS}(\ch)}=\tr{(S^\ast T)}.
\end{flalign*}
For more details, refer to \cite{Fol:94}. 

Let $G$ be a second countable, non-abelian, unimodular and type I group. Let $dx$ be the Haar measure on $G$. Let $d\pi$ be the Plancherel meaure on $\widehat{G}$. For each $(x,\pi)\in G\times \widehat{G}$, let 
\begin{flalign*}
\ch_{(x,\pi)}=\pi(x)\text{HS}(\ch_\pi),
\end{flalign*}
where $\pi(x)\text{HS}(\ch_\pi)=\{\pi(x)T:T\in \text{HS}(\ch_\pi)\}$. Then, $\ch_{(x,\pi)}$ is a Hilbert space with the inner product given by
\begin{flalign*}
\langle \pi(x)T,\pi(x)S\rangle_{\ch_{(x,\pi)}}=\tr{(S^\ast T)}=\langle T,S\rangle_{\text{HS}(\ch_\pi)}.
\end{flalign*}
One can easily verify that $\ch_{(x,\pi)}=\text{HS}(\ch_\pi)$ for all $(x,\pi)\in G\times \widehat{G}$. The family $\{\ch_{(x,\pi)}\}_{(x,\pi)\in G\times \widehat{G}}$ of Hilbert spaces indexed by $G\times \widehat{G}$ is a field of Hilbert spaces over $G\times \widehat{G}$. Let $\ch^2(G \times \widehat{G})$ denote the direct integral of $\{\ch_{(x,\pi)}\}_{(x,\pi)\in G\times \widehat{G}}$ with respect to the product measure $dx\ d\pi$, i.e., the space of all measurable vector fields $F$ on $G\times \widehat{G}$ such that
\begin{flalign*}
\|F\|_{\ch^2(G\times \widehat{G})}^2=\int_{G\times \widehat{G}}{\|F(x,\pi)\|_{(x,\pi)}^2}\ dx\ d\pi<\infty.
\end{flalign*}
$\ch^2(G \times \widehat{G})$ is a Hilbert space with the inner product given by
\begin{flalign*}
\langle F,K\rangle_{\ch^2(G\times \widehat{G})}=\int_{G\times \widehat{G}}{\tr{[F(x,\pi)K(x,\pi)^\ast]}}\ dx\ d\pi.
\end{flalign*}

Let $f \in C_c(G)$, the set of all continuous complex-valued functions on $G$ with compact supports and $\psi$ be a fixed non-zero function in $L^2(G)$ which is sometimes called a window function. For $(x,\pi) \in G\times \widehat{G}$, the continuous \textit{Gabor Transform} of $f$ with respect to the window function $\psi$ can be defined as a measurable field of operators on $G\times \widehat{G}$ by
\begin{flalign}
&&G_\psi f(x,\pi)&:=\int_{G}{f(y)\ \overline{\psi(x^{-1}y)}\ \pi(y)^\ast}\ dy. \label{opvalued}&
\end{flalign}
The operator-valued integral \eqref{opvalued} is considered in the weak-sense, i.e., for each $(x,\pi) \in G\times \widehat{G}$ and $\xi, \eta\in \ch_\pi$, we have
\begin{flalign*}
\langle G_\psi f(x,\pi)\xi,\eta\rangle &=\int_{G}{f(y)\ \overline{\psi(x^{-1}y)}\ \langle\pi(y)^\ast \xi,\eta\rangle}\ dy. 
\end{flalign*}
For each $x\in G$, define $f^\psi_{x}:G \rightarrow \C$ by
\begin{flalign*}
f^\psi_x(y):=f(y)\ \overline{\psi(x^{-1}y)}. 
\end{flalign*}
Since, $f \in C_c(G)$ and $\psi \in L^2(G)$, we have $f^\psi_x\in L^1(G)\cap L^2(G)$, for all $x \in G$. The Fourier transform is given by
\begin{flalign*}
&&\widehat{f^\psi_x}(\pi)&=\int_{G}{f^\psi_x(y)\ \pi(y)^\ast}\ dy =\int_{G}{f(y)\ \overline{\psi(x^{-1}y)}\ \pi(y)^\ast}\ dy =G_\psi f(x,\pi). &
\end{flalign*}
Also, using Plancherel theorem \cite[Theorem $7.44$]{Fol:94}, we see that $\widehat{f^\psi_x}(\pi)$ is a Hilbert-Schmidt operator for almost all $\pi\in \widehat{G}$. Therefore, $G_\psi f(x,\pi)$ is a Hilbert-Schmidt operator for all $x \in G$ and for almost all $\pi\in \widehat{G}$. As in \cite{Far:Kam:12}, for $f \in C_c(G)$ and a window function $\psi \in L^2(G)$, we have
\begin{flalign*}
\|G_\psi f\|_{\ch^2(G\times \widehat{G})} =\|\psi\|_2\ \|f\|_2.
\end{flalign*}
The above equality shows that the continuous Gabor transform $G_\psi : C_c(G) \rightarrow \ch^2(G\times \widehat{G})$ defined by $f \mapsto G_\psi f$ is a multiple of an isometry. So, we can extend $G_\psi$ uniquely to a bounded linear operator from $L^2(G)$ into a closed subspace $H$ of $\ch^2(G\times \widehat{G})$ which we still denote by $G_\psi$ and this extension satisfies
\begin{flalign}
&&\|G_\psi f\|_{\ch^2(G\times \widehat{G})} &=\|\psi\|_2\ \|f\|_2, \label{GT-norm}&
\end{flalign}
for each $f\in L^2(G)$. We now prove an important lemma.
\begin{lem}\label{GT-relation}
Let $f \in L^2(G)$ and $\psi \in L^2(G)$ be a window function. Then
\begin{flalign*}
G_\psi f(x,\pi)=\widehat{f^\psi_x}(\pi).
\end{flalign*}
\end{lem}
\begin{proof}
Let $f \in L^2(G)$. Since $C_c(G)$ is dense in $L^2(G)$, there exists a sequence $\{\phi_n\}$ in $C_c(G)$ such that $f=\displaystyle\lim_{n\rightarrow \infty}{\phi_n}$ in the $L^2$-norm. It follows that
\begin{flalign*}
G_\psi : L^2(G) \rightarrow H \subseteq \ch^2(G\times \widehat{G}) 
\end{flalign*}
satisfies $G_\psi f= \lim\limits_{n\rightarrow \infty}{G_\psi\phi_n}$ in the $\ch^2(G\times \widehat{G})$-norm and 
\begin{flalign*}
G_\psi \phi_n(x,\pi)=\widehat{(\phi_n)^\psi_x}(\pi).
\end{flalign*}
\begin{flalign*}
&\text{Now,}&\|G_\psi f-G_\psi \phi_n\|_{\ch^2(G\times \widehat{G})}^2&=\int_{G}\int_{\widehat{G}}{\|G_\psi f(x,\pi)-G_\psi \phi_n(x,\pi)\|_{\text{HS}}^2}\ dx\ d\pi &\\
&&&=\int_{G}\int_{\widehat{G}}{\|G_\psi f(x,\pi)-\widehat{(\phi_n)^\psi_x}(\pi)\|_{\text{HS}}^2}\ dx\ d\pi &
\end{flalign*}
\begin{flalign*}
&\text{and}&\|\psi\|_2^2\ \|f-\phi_n\|_2^2 &=\int_{G}{|\psi(x)|^2}\ dx\ \int_{G}{|(f-\phi_n)(y)|^2}\ dy &\\
&&&=\int_{G}\int_{G}{|(f-\phi_n)(y)|^2\ |\overline{\psi(x^{-1}y)}|^2}\ dx\ dy &\\
&&&=\int_{G}\int_{G}{|f(y)\ \overline{\psi(x^{-1}y)}-\phi_n(y)\ \overline{\psi(x^{-1}y)}|^2}\ dx\ dy &\\
&&&=\int_{G}\int_{G}{|(f^\psi_x-(\phi_n)^\psi_x)(y)|^2}\ dx\ dy &\\
&&&=\int_{G}\int_{\widehat{G}}{\|\widehat{f^\psi_x}(\pi)-\widehat{(\phi_n)^\psi_x}(\pi)\|_{\text{HS}}^2}\ dx\ d\pi. &
\end{flalign*}
Hence, $G_\psi f(x,\pi)=\widehat{f^\psi_x}(\pi)$ for all $f \in L^2(G)$.
\end{proof}

We now establish Heisenberg uncertainty inequality for Gabor transform. Let $G=K\ltimes \R^n$, where $K$ is a separable unimodular locally compact group of type I. The continuous \textit{Gabor Transform} of $f$ with respect to the window function $\psi$ can be defined as follows:
\begin{flalign}
&&G_\psi f(u,t,\gamma,\sigma)&:=\int_{G}{f^\psi_{u,t}(k,x)\ \pi_{\gamma,\sigma}(k,x)^\ast}\ dx\ dk, \label{opvalued-CE}&
\end{flalign}
where $f^\psi_{u,t}(k,x)=f(k,x)\ \overline{\psi(ku^{-1},x-t)}$, $(u,t)\in G$, $\gamma\in \widehat{\R^n}$ and $\sigma \in \widehat{K_\gamma}$. Also, the equality in Lemma \ref{GT-relation} takes the following form:
\begin{flalign}
G_\psi f(u,t,\gamma,\sigma)=\pi_{\gamma,\sigma}(f^\psi_{u,t}). \label{GT-relation-CE}
\end{flalign}
\begin{thm}
Let $G=K\ltimes \R^n$ satisfies the inequality \eqref{heq-FT-G} and $\psi$ be a window function. For $a,b \geq 1$, we have 
\begin{flalign}
&\|\psi\|_2^{\frac{1}{b}}\ \|f\|_2^{\left(\frac{1}{a}+\frac{1}{b}\right)}\leq C\left(\displaystyle\int_{K\times \R^n}{\|x\|^{2a}\ |f(k,x)|^2}\ dx\ dk\right)^{\frac{1}{2a}}\nonumber&\\
&\qquad \times \left(\displaystyle\int_{K\times \R^n}\displaystyle\int_{\widehat{\R^n}/G}\displaystyle\int_{\widehat{K_\gamma}}{\|\gamma\|^{2b}\ \|G_\psi f(u,t,\gamma,\sigma)\|_{\text{HS}}^2}\ d\mu_\gamma(\sigma)\ d\overline{\mu}_{\R^n}(\overline{\gamma})\ du\ dt\right)^{\frac{1}{2b}}. \label{heq-GT-CE}&
\end{flalign}
\end{thm}
\begin{proof}
Assume that both integrals on the right-hand side of \eqref{heq-GT-CE} are finite. Since $f^\psi_{u,t} \in L^2(G)$ for all $(u,t) \in G$, so by using inequality \eqref{heq-FT-G} for $a=b=1$, we have
\begin{flalign}
\|f^\psi_{u,t}\|_2^2 &\leq C\left(\displaystyle\int_{K\times \R^n}{\|x\|^2\ |f^\psi_{u,t}(k,x)|^2}\ dx\ dk\right)^{1/2} \nonumber &\\
&\qquad \times \left(\displaystyle\int_{\widehat{\R^n}/G}\displaystyle\int_{\widehat{K_\gamma}}{\|\gamma\|^2\ \|\pi_{\gamma,\sigma}(f^\psi_{u,t})\|_{\text{HS}}^2}\ d\mu_\gamma(\sigma)\ d\overline{\mu}_{\R^n}(\overline{\gamma})\right)^{1/2}. \label{step1-CE}&
\end{flalign}
Also, by Proposition \ref{PT-CE} and \eqref{GT-relation-CE}, we have
\begin{flalign}
&\displaystyle\int_{\widehat{\R^n}/G}\displaystyle\int_{\widehat{K_\gamma}}{\|G_\psi f(u,t,\gamma,\sigma)\|_{\text{HS}}^2}\ d\mu_\gamma(\sigma)\ d\overline{\mu}_{\R^n}(\overline{\gamma}) \nonumber&\\*
&=\displaystyle\int_{\widehat{\R^n}/G}\displaystyle\int_{\widehat{K_\gamma}}{\|\pi_{\gamma,\sigma}(f^\psi_{u,t})\|_{\text{HS}}^2}\ d\mu_\gamma(\sigma)\ d\overline{\mu}_{\R^n}(\overline{\gamma}) =\|f^\psi_{u,t}\|_2^2. \label{step2-CE}&
\end{flalign}
On combining \eqref{step1-CE} and \eqref{step2-CE}, we obtain
\begin{flalign*}
&\displaystyle\int_{\widehat{\R^n}/G}\displaystyle\int_{\widehat{K_\gamma}}{\|G_\psi f(u,t,\gamma,\sigma)\|_{\text{HS}}^2}\ d\mu_\gamma(\sigma)\ d\overline{\mu}_{\R^n}(\overline{\gamma}) &\\
& \leq C\left(\displaystyle\int_{K\times \R^n}{\|x\|^2\ |f^\psi_{u,t}(k,x)|^2}\ dx\ dk\right)^{1/2}&\\
&\qquad \times \left(\displaystyle\int_{\widehat{\R^n}/G}\displaystyle\int_{\widehat{K_\gamma}}{\|\gamma\|^2\ \|\pi_{\gamma,\sigma}(f^\psi_{u,t})\|_{\text{HS}}^2}\ d\mu_\gamma(\sigma)\ d\overline{\mu}_{\R^n}(\overline{\gamma})\right)^{1/2}, 
\end{flalign*}
which holds for almost all $(u,t) \in G$. Integrating both sides with respect to $du\ dt$ and then applying Cauchy-Schwarz inequality, we have
\begin{flalign*}
&\displaystyle\int_{K\times \R^n}\displaystyle\int_{\widehat{\R^n}/G}\displaystyle\int_{\widehat{K_\gamma}}{\left\|G_\psi f(u,t,\gamma,\sigma)\right\|^2}\ d\mu_\gamma(\sigma)\ d\overline{\mu}_{\R^n}(\overline{\gamma})\ du\ dt &\\
& \leq C\left(\displaystyle\int_{K\times \R^n}\displaystyle\int_{K\times \R^n}{\|x\|^2\ |f^\psi_{t,u}(k,x)|^2}\ dx\ dk\ du\ dt\right)^{1/2} &\\
&\qquad \times \left(\displaystyle\int_{K\times \R^n}\displaystyle\int_{\widehat{\R^n}/G}\displaystyle\int_{\widehat{K_\gamma}}{\|\gamma\|^2\ \|\pi_{\gamma,\sigma}(f^\psi_{u,t})\|_{\text{HS}}^2}\ d\mu_\gamma(\sigma)\ d\overline{\mu}_{\R^n}(\overline{\gamma})\ du\ dt\right)^{1/2}  &\\
&=C \left(\displaystyle\int_{K\times \R^n}\displaystyle\int_{K\times \R^n}{\|x\|^2\ |f(k,x)\ \overline{\psi(ku^{-1},x-t)}|^2}\ dx\ dk\ du\ dt\right)^{1/2} &\\*
&\qquad \times \left(\displaystyle\int_{K\times \R^n}\displaystyle\int_{\widehat{\R^n}/G}\displaystyle\int_{\widehat{K_\gamma}}{\|\gamma\|^2\ \|\pi_{\gamma,\sigma}(f^\psi_{u,t})\|_{\text{HS}}^2}\ d\mu_\gamma(\sigma)\ d\overline{\mu}_{\R^n}(\overline{\gamma})\ du\ dt\right)^{1/2}  &\\
&= C\|\psi\|_2\left(\displaystyle\int_{K\times \R^n}{\|x\|^2\ |f(k,x)|^2}\ dx\ dk\right)^{1/2} &\\*
&\qquad \times\left(\displaystyle\int_{K\times \R^n}\displaystyle\int_{\widehat{\R^n}/G}\displaystyle\int_{\widehat{K_\gamma}}{\|\gamma\|^2\ \|\pi_{\gamma,\sigma}(f^\psi_{u,t})\|_{\text{HS}}^2}\ d\mu_\gamma(\sigma)\ d\overline{\mu}_{\R^n}(\overline{\gamma})\ du\ dt\right)^{1/2}.  &
\end{flalign*}
Using \eqref{GT-norm} and \eqref{GT-relation-CE}, we get
\begin{flalign}
&\|\psi\|_2\ \|f\|_2^2\nonumber &\\*
&\leq C\left(\displaystyle\int_{K\times \R^n}{\|x\|^2\ |f(k,x)|^2}\ dx\ dk\right)^{1/2} \nonumber&\\
&\qquad \times\left(\displaystyle\int_{K\times \R^n}\displaystyle\int_{\widehat{\R^n}/G}\displaystyle\int_{\widehat{K_\gamma}}{\|\gamma\|^2\ \|G_\psi f(u,t,\gamma,\sigma)\|_{\text{HS}}^2}\ d\mu_\gamma(\sigma)\ d\overline{\mu}_{\R^n}(\overline{\gamma})\ du\ dt\right)^{1/2}.  &\label{step3-CE}
\end{flalign}
Applying H\"older's inequality, we have
\begin{flalign}
&\left(\displaystyle\int_{K\times \R^n}{\|x\|^{2a}\ |f(k,x)|^2}\ dx\ dk\right)^{\frac{1}{a}}\left(\displaystyle\int_{K\times \R^n}{|f(k,x)|^2}\ dx\ dk\right)^{1-\frac{1}{a}}\nonumber& \\
&\geq \displaystyle\int_{K\times \R^n}{\|x\|^2\ |f(k,x)|^2}\ dx\ dk   \label{step4-CE}&
\end{flalign}
and
\begin{flalign}
&\left(\displaystyle\int_{K\times \R^n}\displaystyle\int_{\widehat{\R^n}/G}\displaystyle\int_{\widehat{K_\gamma}}{\|\gamma\|^{2b}\ \|G_\psi f(u,t,\gamma,\sigma)\|_{\text{HS}}^2}\ d\mu_\gamma(\sigma)\ d\overline{\mu}_{\R^n}(\overline{\gamma})\ du\ dt\right)^{\frac{1}{b}}\nonumber&\\*
&\qquad \times\left(\displaystyle\int_{K\times \R^n}\displaystyle\int_{\widehat{\R^n}/G}\displaystyle\int_{\widehat{K_\gamma}}{\|G_\psi f(u,t,\gamma,\sigma)\|_{\text{HS}}^2}\ d\mu_\gamma(\sigma)\ d\overline{\mu}_{\R^n}(\overline{\gamma})\ du\ dt\right)^{1-\frac{1}{b}}\nonumber&\\
&\geq\displaystyle\int_{K\times \R^n}\displaystyle\int_{\widehat{\R^n}/G}\displaystyle\int_{\widehat{K_\gamma}}{\|\gamma\|^2\ \|G_\psi f(u,t,\gamma,\sigma)\|_{\text{HS}}^2}\ d\mu_\gamma(\sigma)\ d\overline{\mu}_{\R^n}(\overline{\gamma})\ du\ dt. \label{step5-CE}& 
\end{flalign}
Combining \eqref{step3-CE}, \eqref{step4-CE} and \eqref{step5-CE}, we have
\begin{flalign*}
&\|\psi\|_2\ \|f\|_2^2&\\
&\leq C\left(\displaystyle\int_{K\times \R^n}{\|x\|^{2a}\ |f(k,x)|^2}\ dx\ dk\right)^{\frac{1}{2a}}\ (\|f\|_2^2)^{\frac{1}{2}-\frac{1}{2a}} &\\*
&\qquad \times\left(\displaystyle\int_{K\times \R^n}\displaystyle\int_{\widehat{\R^n}/G}\displaystyle\int_{\widehat{K_\gamma}}{\|\gamma\|^{2b}\ \|G_\psi f(u,t,\gamma,\sigma)\|_{\text{HS}}^2}\ d\mu_\gamma(\sigma)\ d\overline{\mu}_{\R^n}(\overline{\gamma})\ du\ dt\right)^{\frac{1}{2b}}&\\
&\qquad \times(\|\psi\|_2^2\ \|f\|_2^2)^{\frac{1}{2}-\frac{1}{2b}}. &
\end{flalign*}
Thus, we have the required inequality \eqref{heq-GT-CE}.
\end{proof}
\begin{ex}
We give the explicit expression of the Heisenberg uncertainty inequality for Gabor transform in the following cases:
\begin{enumerate}
\item Euclidean group $\R^n$.
\begin{flalign*}
&\|\psi\|_2^{\frac{1}{b}}\ \|f\|_2^{\left(\frac{1}{a}+\frac{1}{b}\right)}\leq C\left(\displaystyle\int_{\R^n}{\|x\|^{2a}\ |f(x)|^2}\ dx\right)^{\frac{1}{2a}}\\*
&\qquad \times \left(\displaystyle\int_{\R^n}\displaystyle\int_{\widehat{\R^n}}{\|\omega\|^{2b}\ \|G_\psi f(t,\omega)\|_{\text{HS}}^2}\ dt\ d\omega\right)^{\frac{1}{2b}}. 
\end{flalign*}
\item $\R^n\times K$, where $K$ is a separable unimodular locally compact group of type I.
\begin{flalign*}
&\|\psi\|_2^{\frac{1}{b}}\ \|f\|_2^{\left(\frac{1}{a}+\frac{1}{b}\right)}\leq C\left(\displaystyle\int_{\R^n\times K}{\|x\|^{2a}\ |f(x,k)|^2}\ dx\ dk\right)^{\frac{1}{2a}}\\*
&\qquad \times \left(\displaystyle\int_{\R^n\times K}\displaystyle\int_{\R^n\times \widehat{K}}{\|z\|^{2b}\ \|G_\psi f(t,u,z,\gamma)\|_{\text{HS}}^2}\ dz\  d\gamma\ dt\ du\right)^{\frac{1}{2b}}. 
\end{flalign*}
\item Heisenberg Group $\h_n$ (see \cite{Tha:90}).
\begin{flalign*}
&\|\psi\|_2^{\frac{1}{b}}\ \|f\|_2^{\left(\frac{1}{a}+\frac{1}{b}\right)}\leq C\left(\displaystyle\int_{\h_n}{|t|^{2a}\ |f(z,t)|^2}\ dz\ dt\right)^{\frac{1}{2a}} \\*
&\qquad \times \left(\displaystyle\int_{\h_n}\displaystyle\int_{\R^\ast}{|\lambda|^{2b}\ \|G_\psi f(z',t',\lambda)\|_{\text{HS}}^2\ |\lambda|^n}\ d\lambda\ dz'\ dt'\right)^{\frac{1}{2b}}. 
\end{flalign*}
\item $K \ltimes \R^n$, where $K$ is a compact subgroup of the group of automorphisms of $\R^n$.
\begin{flalign*}
&\|\psi\|_2^{\frac{1}{b}}\ \|f\|_2^{\left(\frac{1}{a}+\frac{1}{b}\right)}\leq C\left(\displaystyle\int_{K\times \R^n}{\|x\|^{2a}\ |f(k,x)|^2}\ dx\ dk\right)^{\frac{1}{2a}}\\*
&\qquad \times \left(\displaystyle\int_{K\times \R^n}\displaystyle\int_{\widehat{\R^n}/G}\sum_{\sigma \in \widehat{K_\ell}}{\|\ell\|^{2b}\ \|G_\psi f(u,t,\ell,\sigma)\|_{\text{HS}}^2}\ d\bar{\ell}\ du\ dt\right)^{\frac{1}{2b}}.
\end{flalign*}
\item A class of connected, simply connected nilpotent Lie groups $G$ for which the Hilbert-Schmidt norm of the group Fourier transform $\pi_\xi(f)$ of $f$ attains a particular form (see \cite{Ban:Kum:15}).
\begin{flalign*}
&\|\psi\|_2^{\frac{1}{b}}\ \|f\|_2^{\left(\frac{1}{a}+\frac{1}{b}\right)}\leq C\left(\displaystyle\int_{G}{\|x\|^{2a}\ |f(x)|^2}\ dx\right)^{\frac{1}{2a}}\\*
&\qquad \times \left(\displaystyle\int_{G}\displaystyle\int_{\W}{\|\xi\|^{2b}\ \|G_\psi f(y,\xi)\|_{\text{HS}}^2\ \dfrac{1}{|h(\xi)|^b\ |\pf(\xi)|^{b-1}}}\ d\xi\ dy\right)^{\frac{1}{2b}}. 
\end{flalign*}
\item For thread-like nilpotent Lie groups (see \cite{Kan:Kum:01}).
\begin{flalign*}
&\|\psi\|_2^{\frac{1}{b}}\ \|f\|_2^{\left(\frac{1}{a}+\frac{1}{b}\right)}\leq C\left(\displaystyle\int_{G}{\|x\|^{2a}\ |f(x)|^2}\ dx\right)^{\frac{1}{2a}}\\*
&\qquad\times \left(\displaystyle\int_{G}\displaystyle\int_{\W}{\|\xi\|^{2b}\ \|G_\psi f(y,\xi)\|_{\text{HS}}^2}\ |\xi_1|\ d\xi\right)^{\frac{1}{2b}}.
\end{flalign*}
\item For $2$-NPC nilpotent Lie groups (see \cite{Bak:Sal:08}), let $\{0\}=\g_0 \subset \g_1 \subset \cdots \subset \g_n =\g$ be a Jordan-H\"older sequence in $\g$ such that $\g_m=\z(g)$ and $\hf=\g_{n-2}$. We have the following two cases:
\begin{enumerate}
\item $\dim{[\g,\g_{m+1}]}=2$. 
\begin{flalign*}
&\|\psi\|_2^{\frac{1}{b}}\ \|f\|_2^{\left(\frac{1}{a}+\frac{1}{b}\right)}\leq C\left(\displaystyle\int_{G}{\|x\|^{2a}\ |f(x)|^2}\ dx\right)^{\frac{1}{2a}}\\
&\qquad \times \left(\displaystyle\int_{G}\displaystyle\int_{\W}{\|\xi\|^{2b}\ \|G_\psi f(y,\xi)\|_{\text{HS}}^2}\ \dfrac{1}{|h(\xi)|^b|\pf(\xi)|^{b-1}}\ d\xi\right)^{\frac{1}{2b}}.
\end{flalign*}
\item $\dim{[\g,\g_{m+1}]}=1$. 
\begin{flalign*}
&\|\psi\|_2^{\frac{1}{b}}\ \|f\|_2^{\left(\frac{1}{a}+\frac{1}{b}\right)}\leq C\left(\displaystyle\int_{G}{\|x\|^{2a}\ |f(x)|^2}\ dx\right)^{\frac{1}{2a}}\\
&\times \left(\displaystyle\int_{G}\displaystyle\int_{\W}{\|\xi\|^{2b}\ \|G_\psi f(y,\xi)\|_{\text{HS}}^2}\ |\pf(\xi)|\ d\xi\right)^{\frac{1}{2b}}.
\end{flalign*}
\end{enumerate}
\item For connected simply connected nilpotent Lie groups $G=\expo{\g}$ such that $\g(\xi) \subset [\g,\g]$ for all $\xi \in \U$ (see \cite{Sma:11}).
\begin{flalign*}
&\|\psi\|_2^{\frac{1}{b}}\ \|f\|_2^{\left(\frac{1}{a}+\frac{1}{b}\right)} \leq C\left(\displaystyle\int_{G}{\|x\|^{2a}\ |f(x)|^2}\ dx\right)^{\frac{1}{2a}}\\
&\times\left(\displaystyle\int_{G}\displaystyle\int_{\W}{\|\xi\|^{2b}\ \|G_\psi f(y,\xi)\|_{\text{HS}}^2}\ \dfrac{|\pf(\xi)|^{b+1}}{|\xi([X_{j_1},X_n])|^b}\ d\xi\right)^{\frac{1}{2b}}.
\end{flalign*}
\item For low-dimensional nilpotent Lie groups of dimension less than or equal to $6$ (for details, see \cite{Nie:83}) except for $G_{6,8}$, $G_{6,12}$, $G_{6,14}$, $G_{6,15}$, $G_{6,17}$, one can write an explicit Heisenberg uncertainty inequality for Gabor transform.
\end{enumerate}
\end{ex}

\end{document}